\documentclass{amsproc}

\usepackage[latin1]{inputenc}
\usepackage{amsfonts}
\usepackage{amsthm}
\usepackage{amsmath}
\usepackage{amssymb}
\usepackage{amscd}
\usepackage{verbatim}
\usepackage{multicol}
\usepackage{tabularx}

\theoremstyle{definition}
\newtheorem{thm}{Theorem}[section]
\newtheorem{prop}[thm]{Proposition}
\newtheorem{cor}[thm]{Corollary}

\newtheorem{lem}[thm]{Lemma}
\newtheorem{defn}[thm]{Definition}

\newtheorem*{rem}{Remark}

\numberwithin{equation}{section}

\def\cal#1{\text{$\mathcal{#1}$}}
\def\lie#1{\mathfrak{#1}}
\def\tlie#1{\tilde{\mathfrak{#1}}}

\def\gb#1{{\mbox{\boldmath $#1$}}}

\def\wt{{\rm wt}}
\def\chl{{\rm char}_\ell}
\def\ch{{\rm char}}
\def\opl_#1^#2{\text{\scriptsize$\bigoplus\limits_{\text{\footnotesize$#1$}}^{\text{\footnotesize$#2$}}$}}
\def\otm_#1^#2{\text{\scriptsize$\bigotimes\limits_{\text{\footnotesize$#1$}}^{\text{\footnotesize$#2$}}$}}

\def\mult^#1_#2{{\rm mult}^{#1}_{#2}}
\def\ideg{{\rm indeg}}

\renewcommand{\thefootnote}

\begin{document}

\flushbottom

\title[On Multiplicity Problems for Hyper Loop Algebras]{On Multiplicity Problems for\\ Finite-Dimensional Representations of\\ Hyper Loop Algebras}

\author[D. Jakeli\'c]{Dijana Jakeli\'c}
\address{Department of Mathematics, Statistics, and Computer Science, \vspace{-8pt}}
\address{University of Illinois at Chicago, Chicago, IL 60607-7045, USA} \email{dijana@math.uic.edu}

\author[A. Moura]{Adriano Moura}
\address{UNICAMP - IMECC, Campinas - SP, 13083-970, Brazil.} \email{aamoura@ime.unicamp.br}
\thanks{The research of A.M. is partially supported by CNPq and FAPESP. A.M also thanks the Department of Mathematics, Statistics, and Computer Science of the University of Illinois at Chicago for its hospitality and support during the period this note was written.}


\begin{abstract}
Given a hyper loop algebra over a non-algebraically closed field, we address multiplicity problems in the underlying abelian tensor category of finite-dimensional representations. Namely, we give formulas for the $\ell$-characters of the simple objects,  the Jordan-H\"older multiplicities of the Weyl modules, and the  Clebsch-Gordan coefficients.
\end{abstract}

\maketitle

\section*{Introduction}
Let $\cal C$ be a Jordan-H\"older tensor category. Roughly speaking, this is an abelian category with finite-length  objects equipped with a tensor product (for a precise definition of abelian and tensor categories see,  for instance, \cite[Chapter 5]{cpbook}). Given two objects $V$ and $W$ of $\cal C$ with $V$ simple, let $\mult^{\cal C}_V(W)$ be the multiplicity of $V$ as an irreducible constituent of $W$ in the category $\cal C$. We will often refer to $\mult^{\cal C}_V(W)$ as the Jordan-H\"older multiplicity of $V$ inside $W$ in the category $\cal C$. Formulas for Jordan-H\"older multiplicities make one of the most exciting subjects in combinatorial representation theory. For instance, if $\lie g$ is a Kac-Moody algebra, the Jordan-H\"older multiplicities for an integrable module $W$ may be computed from the character of $W$. In order to do that one has to know the characters of the irreducible integrable modules. These are, in turn, given by the celebrated Weyl-Kac character formula which, in some special cases, may be used to recover some famous identities including Rogers-Ramanujan's identities (see \cite{clm,lw} and references therein). An important subclass of multiplicity problems is the one of computing Clebsch-Gordan coefficients, i.e., of computing $\mult^{\cal C}_V(W)$  in a given category $\cal C$ when $W$ is a tensor product of two simple objects. In the case of a Kac-Moody algebra these coefficients may, in principle, be found as a consequence of the Weyl-Kac character formula.

In this note we address some of the aforementioned  problems for the category $\cal C(\tlie g)_\mathbb K$ of finite-dimensional representations of the hyper loop algebra $U(\tlie g)_\mathbb K$ over a field $\mathbb K$. Here, $\lie g$ denotes a fixed simple finite-dimensional Lie algebra over the complex numbers. The finite-dimensional representation theory of hyper loop algebras was initiated in \cite{jmhla,jmnacf}. As remarked in the introduction of \cite{jmnacf}, in an appropriate algebraic geometric framework, it should turn out that studying finite-dimensional representations of these algebras is equivalent to studying finite-dimensional representations of certain algebraic loop groups (in the defining characteristic).
We remark that the hyperalgebra $U(\lie g)_\mathbb K$ of the underlying Chevalley group of universal type is a subalgebra of $U(\tlie g)_\mathbb K$.

It was shown in \cite{jmnacf} that the passage from algebraically closed fields to non-algebraically closed ones is much more interesting in the context of hyper loop algebras than in the case of the hyperalgebra $U(\lie g)_\mathbb K$. For instance, multiplicity formulas, such as the ones mentioned in the first paragraph, for the category $\cal C(\lie g)_\mathbb K$ of finite-dimensional $U(\lie g)_\mathbb K$-modules are invariant under field extensions. However, it was already clear from several results and examples in \cite{jmnacf} that these problems have a much richer relation with Galois theory in the case of hyper loop algebras.

The purpose of this note is to solve certain multiplicity problems for $\cal C(\tlie g)_\mathbb K$. Namely, we find formulas for the following: (1) the $\ell$-characters of the simple objects,  (2) the Jordan-H\"older multiplicities of the Weyl modules, and (3) the Clebsch-Gordan coefficients. By a solution we mean the following. Suppose $\mathbb K$ is a non-algebraically closed field and $\mathbb F$ is an algebraic closure of $\mathbb K$. Our formulas express the multiplicities for the category $\cal C(\tlie g)_\mathbb K$ in terms of the corresponding multiplicities for the category $\cal C(\tlie g)_\mathbb F$ via Galois theory. We also review the literature related to these problems in the algebraically closed setting. It turns out that the multiplicities in the category $\cal C(\tlie g)_\mathbb F$ may be expressed in terms of the corresponding formulas for the category $\cal C(\lie g)_\mathbb F$. We give bibliographic references for the later and remark that some of these problems are still open in the category $\cal C(\lie g)_\mathbb F$ when $\mathbb F$ has positive characteristic.

The paper is organized as follows. In \S\ref{s:pre}, we review several prerequisites such as the construction of the hyperalgebras and hyper loop algebras as well as the relevant results about their finite-dimensional representations. In \S\ref{s:cg}, we present the results about the Clebsch-Gordan problem. In \S\ref{s:jh}, we study the $\ell$-characters of the simple objects of $\cal C(\tlie g)_\mathbb K$. This is equivalent to computing their Jordan-H\"older multiplicities in the category of finite-dimensional representations of a certain polynomial subalgebra of $U(\tlie g)_\mathbb K$. We end the section by expressing the results of \S 2 in terms of $\ell$-characters.
Finally, \S\ref{s:jhw} is concerned with the computation the Jordan-H\"older multiplicities for the Weyl modules.

\section{Preliminaries}\label{s:pre}

\subsection{}

 Throughout the paper, let $\mathbb F$ be a fixed algebraically closed field and $\mathbb K$ a  subfield of  $\mathbb F$ having $\mathbb F$ as an algebraic closure. Also,
$\mathbb C, \mathbb Z,\mathbb Z_+,\mathbb N$ denote the sets of complex numbers, integers, non-negative integers, and positive integers, respectively. Given a ring $\mathbb A$, the underlying multiplicative group of units is denoted by $\mathbb A^\times$. The dual of a vector space $V$ is denoted by $V^*$. The symbol $\cong$ means ``isomorphic to''.

If $\mathbb L$ is a field and $A$ an associative  unitary  $\mathbb L$-algebra, the expression ``$V$ is an $A$-module'' means $V$ is an $\mathbb L$-vector space equipped with an $A$-action, i.e., a homomorphism of associative unitary $\mathbb L$-algebras $A\to {\rm End}_{\mathbb L}(V)$.  The action of $a\in A$ on $v\in V$ is denoted simply by $av$.
If $V$ is an $\mathbb L$-vector space and $\mathbb M$ is a field extension of $\mathbb L$, set $V^{\mathbb M}=\mathbb M\otimes_{\mathbb L} V$ and identify $V$ with $1\otimes V\subseteq V^\mathbb M$. If $V$  is an $A$-module, then $V^\mathbb M$ is naturally an $A^\mathbb M$-module.

The next lemma is immediate and will be repeatedly used in the proofs of the main results.

\begin{lem}\label{l:compseries}
Let $A$ be a $\mathbb K$-algebra, $V$ an $A$-module, and $0\subseteq V_1\subseteq V_2\subseteq\cdots\subseteq V_m=V$ a composition series for $V$. Then $0\subseteq V_1^\mathbb F\subseteq V_2^\mathbb F\subseteq\cdots\subseteq V_m^\mathbb F=V^\mathbb F$ is a sequence of inclusions of $A^\mathbb F$-modules and $V_j^\mathbb F/V_{j-1}^\mathbb F\cong (V_j/V_{j-1})^\mathbb F$.\hfill\qedsymbol
\end{lem}

\subsection{}  Given two $\mathbb K$-vector spaces $V$ and $W$, let $\varphi^\mathbb F_{V,W}:V^\mathbb F\otimes_\mathbb F W^\mathbb F\to (V\otimes_\mathbb K W)^\mathbb F$ be the isomorphism of $\mathbb F$-vector spaces determined by  $(a\otimes v)\otimes(b\otimes w)\mapsto (ab)\otimes(v\otimes w)$.

Recall that if $A$ is a Hopf algebra the tensor product of $A$-modules can be equipped with a structure of $A$-module by using the comultiplication. The Hopf algebra structure of $A$ can be extended to one on $A^\mathbb F$. In particular, the comultiplication is given by $\Delta^\mathbb F(b\otimes x) = (\varphi^\mathbb F_{A,A})^{-1}(b\otimes\Delta(x))$, where $\Delta$ is the comultiplication on $A$.
One now easily checks the following lemma which will be used in the proof of Theorem \ref{t:cg} below.

\begin{lem}\label{l:tpfe}
Let $A$ be a Hopf algebra over $\mathbb K$. If $V,W$ are $A$-modules, $\varphi^\mathbb F_{V,W}$ is an isomorphism of $A^\mathbb F$-modules.\hfill\qedsymbol
\end{lem}

\subsection{} Let $I$ be the set of vertices of a finite-type connected Dynkin
diagram and let $\lie g$ be the corresponding simple Lie algebra over $\mathbb C$ with a fixed Cartan subalgebra $\lie h$.
Fix a set of positive roots $R^+$ and let
$$\lie n^\pm = \opl_{\alpha\in R^+}^{} \lie g_{\pm\alpha} \quad\text{where}\quad \lie g_{\pm\alpha} = \{x\in\lie g: [h,x]=\pm\alpha(h)x, \ \forall \ h\in\lie h\}.$$
Fix also a Chevalley basis $\{x_\alpha^\pm, h_i: \alpha\in R^+, i\in I\}$ for $\lie g$ so that $\lie n^\pm = \mathbb Cx^\pm_\alpha$ and $[x_{\alpha_i}^+,x_{\alpha_i}^-]=h_i$. Here, $\alpha_i \ (i\in I$) denote the simple roots. The fundamental weights are denoted by $\omega_i$ ($i\in I$), while $Q,P,Q^+$, and $P^+$ denote the root and weight lattices with the corresponding positive cones, respectively. We equip $\lie h^*$ with the partial order given by $\lambda\le \mu$ iff $\mu-\lambda\in Q^+$.

Let $\tlie g=\lie g\otimes\mathbb C[t,t^{-1}]=\tlie n^-\oplus\tlie h\oplus\tlie n^+$ be the loop algebra over $\lie g$ with bracket given by $[x\otimes f(t),y\otimes g(t)]=[x,y]\otimes(f(t)g(t))$. Given $\alpha\in R^+,i\in I$, and $r\in\mathbb Z$, let also $x_{\alpha,r}^\pm=x_\alpha^\pm\otimes t^r, h_{i,r}=h_i\otimes t^r$. Clearly, these elements form a basis of $\tlie g$. We identify $\lie g$ with the subalgebra of $\tlie g$ given by the subset $\lie g\otimes 1$ and identify the elements $x_{\alpha,0}^\pm$ and $h_{i,0}$ with the elements $x_{\alpha}^\pm$ and $h_{i}$, respectively.

\subsection{}
For a Lie algebra $\lie a$, let $U(\lie a)$ be its universal enveloping algebra. Consider the $\mathbb Z$-subalgebra $U(\tlie g)_\mathbb Z$ (resp. $U(\lie g)_\mathbb Z$) of $U(\tlie g)$ generated by the elements $(x_{\alpha,r}^\pm)^{(k)}:=\frac{(x_{\alpha,r}^\pm)^k}{k!}$ (resp. $(x_{\alpha}^\pm)^{(k)}:=\frac{(x_{\alpha}^\pm)^k}{k!}$) for all $\alpha\in R^+, r,k\in\mathbb Z, k\ge 0$. $U(\lie g)_\mathbb Z$ is Kostant's integral form of $U(\lie g)$ \cite{kos} and $U(\tlie g)_\mathbb Z$ is Garland's integral form of $U(\tlie g)$ \cite{garala}.
Set $U(\tlie n^\pm)_\mathbb Z=U(\tlie n^\pm)\cap U(\tlie g)_\mathbb Z$ and define $U(\tlie h)_\mathbb Z, U(\lie n^\pm)_\mathbb Z$, and $U(\lie h)_\mathbb Z$ similarly. We have $U(\lie g)_\mathbb Z=U(\lie n^-)_\mathbb ZU(\lie h)_\mathbb ZU(\lie n^+)_\mathbb Z$ and $U(\tlie g)_\mathbb Z=U(\tlie n^-)_\mathbb ZU(\tlie h)_\mathbb ZU(\tlie n^+)_\mathbb Z$.  Moreover, given an ordering on $R^+\times \mathbb Z$, the monomials in the elements $(x_{\alpha,r}^\pm)^{(k)}$ obtained in the obvious way from the corresponding PBW monomials associated to our fixed bases for $\lie n^\pm$ and $\tlie n^\pm$ form bases for $U(\lie n^\pm)_\mathbb Z$ and $U(\tlie n^\pm)_\mathbb Z$, respectively. The set $\{\prod_{i\in I}\binom{h_i}{k_i}: k_i\in\mathbb Z_+\}$ is a basis for $U(\lie h)_\mathbb Z$ where $\binom{h_i}{k} = \frac{1}{k!}h_i(h_i-1)\cdots(h_i-k+1)$. As for $U(\tlie h)_\mathbb Z$, define elements $\Lambda_{i,r}\in U(\tlie h)$ by the following equality of formal power series in the variable $u$:
\begin{equation*}\label{e:Lambdadef}
\sum_{r=0}^\infty \Lambda_{i,\pm r}\ u^r = \exp\left( - \sum_{s=1}^\infty \frac{h_{i,\pm s}}{s}\ u^s\right).
\end{equation*}
It turns out that $\Lambda_{i,r}\in U(\tlie h)_\mathbb Z$ and that the unitary $\mathbb Z$-subalgebras generated by the sets $\gb\Lambda^\pm = \{\Lambda_{i,\pm r}: i\in I,r\in\mathbb N\}$ are the polynomial algebras $\mathbb Z[\gb\Lambda^\pm]$ (cf. \cite[Proposition 1.7]{jmnacf}). Moreover, multiplication defines an isomorphism of $\mathbb Z$-algebras
$$U(\tlie h)_\mathbb Z \cong \mathbb Z[\gb\Lambda^-]\otimes_\mathbb ZU(\lie h)_\mathbb Z\otimes_\mathbb Z\mathbb Z[\gb\Lambda^+].$$

Given a field $\mathbb L$, let $U(\lie g)_\mathbb L$ and $U(\tlie g)_\mathbb L$ be the hyperalgebra and the hyper loop algebra of $\lie g$ over $\mathbb L$ which are defined by
$$U(\lie g)_\mathbb L =  \mathbb L\otimes_{\mathbb Z}U(\lie g)_\mathbb Z \quad\text{and}\quad U(\tlie g)_\mathbb L = \mathbb L\otimes_{\mathbb Z}U(\tlie g)_\mathbb Z.$$
Define also $U(\tlie n^\pm)_\mathbb L$, etc., in the obvious way.

From now on, we set $\gb\Lambda=\gb\Lambda^+$ and let $\mathbb L[\gb\Lambda]=\mathbb L\otimes_\mathbb Z\mathbb Z[\gb\Lambda]\subseteq U(\tlie g)_\mathbb L$. We keep the notation $(x^\pm_{\alpha,r})^{(k)}, \binom{h_i}{k}$, and $\Lambda_{i,r}$ for the images of these elements in $U(\tlie g)_\mathbb L$.

The Hopf algebra structure on $U(\tlie g)$ induces a Hopf algebra structure on $U(\tlie g)_\mathbb L$. We denote the corresponding augmentation ideal by $U(\tlie g)_\mathbb L^0$ (and similarly for the factors of the triangular decompositions above). $U(\lie g)_\mathbb L$ is a Hopf subalgebra of $U(\tlie g)_\mathbb L$.

\begin{rem}
If the characteristic of $\mathbb L$ is zero, the algebra $U(\tlie g)_\mathbb L$ is naturally isomorphic to the enveloping algebra $U(\tlie g_\mathbb L)$ where $\tlie g_\mathbb L = \mathbb L\otimes_\mathbb Z\tlie g_\mathbb Z$ and $\tlie g_\mathbb Z=\tlie g\cap U(\tlie g)_\mathbb Z$ (and similarly for $\lie g$ in place of $\tlie g$).
\end{rem}

For further details regarding this subsection see \cite[\S 1]{jmhla} and \cite[\S 1.3]{jmnacf}.

\subsection{}\label{ss:l-lattices} For a field $\mathbb L$, let $\cal P_\mathbb L^+$ be the monoid of $I$-tuples of polynomials $\gb\omega=(\gb\omega_i(u))$ such that $\gb\omega_i(0)=1$ for all $i\in I$ and let $\cal P_\mathbb L$ be the corresponding abelian group of rational functions. Let also $\gb M_\mathbb L$ be the set of unitary $\mathbb L$-algebra homomorphisms $\gb\varpi:\mathbb L[\gb\Lambda]\to\mathbb L$. Notice that $\gb M_\mathbb L$ can be identified with a submonoid of the multiplicative monoid of $I$-tuples of formal power series by the assignment $\gb\varpi\mapsto (\gb\varpi_i(u))$ where $\gb\varpi_i(u)=1+\sum_{r\in\mathbb N}\varpi_{i,r}u^r$ and $\varpi_{i,r}=\gb\varpi(\Lambda_{i,r})$. If $\mathbb L$ is algebraically closed,  we can identify $\cal P_\mathbb L$ with a subgroup of $\gb M_\mathbb L$ by expanding the denominators of the elements of $\cal P_\mathbb L$ into the corresponding products of geometric power series.

Given $\mu\in P$ and $a\in \mathbb L^\times$, let $\gb\omega_{\mu,a}$ be the element of $\cal P_\mathbb L$ whose $i$-th coordinate rational function is $(\gb\omega_{\mu,a})_i(u) = (1-au)^{\mu(h_i)}$. Let $\cal Q_\mathbb L^+$ be the sub-monoid of $\cal P_\mathbb L$ generated by the elements $\gb\omega_{\alpha,a}$ with $\alpha\in R^+$ and $a\in\mathbb L^\times$ and define a partial order on $\cal P_\mathbb L$ by setting $\gb\varpi\le\gb\omega$ iff $\gb\omega\in\gb\varpi\cal Q_\mathbb L^+$.

\subsection{} Let $\gb\varpi\in\gb M_\mathbb F$.  Define  $\cal F(\gb\varpi)$ to be an $\mathbb F[\gb\Lambda]$-module isomorphic to the quotient of $\mathbb F[\gb\Lambda]$ by the maximal ideal generated by the elements $\Lambda_{i,r}-\varpi_{i,r}$. In particular, $\cal F(\gb\varpi)$ is a one-dimensional $\mathbb F$-vector space. Given a nonzero vector $v\in\cal F(\gb\varpi)$, let $\cal K(\gb\varpi)=\mathbb K[\gb\Lambda]v$ and $\deg(\gb\varpi)=\dim_\mathbb K(\cal K(\gb\varpi))$. It is easy to see that the isomorphism class of $\cal K(\gb\varpi)$ as a $\mathbb K[\gb\Lambda]$-module does not depend on the choice of $v$. Finally, given $g\in{\rm Aut}(\mathbb F/\mathbb K)$, let $g(\gb\varpi)$ be the $I$-tuple of power series whose $i$-th entry is $1+\sum_{r\in\mathbb N} g(\varpi_{i,r})u^r$ and set $[\gb\varpi] = \{g(\gb\varpi): g\in {\rm Aut}(\mathbb F/\mathbb K)\}$. The set $[\gb\varpi]$ is called the conjugacy class of $\gb\varpi$ over $\mathbb K$. Let $\gb M_{\mathbb F,\mathbb K}, \cal P_{\mathbb F,\mathbb K}$, and $\cal P_{\mathbb F,\mathbb K}^+$ be the corresponding sets of conjugacy classes in $\gb M_{\mathbb F}, \cal P_{\mathbb F}$, and $\cal P_{\mathbb F}^+$, respectively . We have (see \cite[Theorem 1.2]{jmnacf}):

{\samepage
\begin{thm}\label{t:irpa}\hfill
\begin{enumerate}
\item $\cal K(\gb\varpi)$ is an irreducible $\mathbb K[\gb\Lambda]$-module for every $\gb\varpi\in\gb M_\mathbb F$ and every finite-dimensional irreducible $\mathbb K[\gb\Lambda]$-module is of this form for some $\gb\varpi\in\gb M_\mathbb F$.
\item $\cal K(\gb\varpi)\cong\cal K(\gb\varpi')$ iff $\gb\varpi'\in[\gb\varpi]$.
\item Suppose $\deg(\gb\varpi)<\infty$. Then $\cal K(\gb\varpi)^\mathbb F\cong \opl_{\gb\varpi'\in[\gb\varpi]}^{} V_{\gb\varpi'}$ with $V_{\gb\varpi'}$  an indecomposable self-extension of  $\cal F(\gb\omega')$ of length $\deg(\gb\omega)/|[\gb\omega]|$, where $|[\gb\omega]|$ is the cardinality of $[\gb\omega]$.\hfill\qedsymbol
\end{enumerate}
\end{thm}}

The number $\deg(\gb\omega)/|[\gb\omega]|$ is called the inseparability degree of $\gb\omega$ over $\mathbb K$ and will be denoted by $\ideg(\gb\omega)$.

\subsection{}\label{ss:fdg} The basic facts about the finite-dimensional representation theory of \linebreak $U(\lie g)_\mathbb K$ are as follows (see \cite[\S 1.4]{jmnacf} and references therein).

\begin{thm}\label{t:cig}
Let $V$ be a finite-dimensional $U(\lie g)_\mathbb K$-module. Then:
\begin{enumerate}
\item $V=\opl_{\mu\in P}^{} V_{\mu}$, where $V_{\mu} = \{v\in v:\binom{h_i}{k}v=\binom{\mu(h_i)}{k}v\}$.
\item If $V$ is irreducible, there exists $\lambda\in P^+$ such that
\begin{equation}\label{e:hwrel}
\dim_\mathbb K(V_\lambda)=1, \qquad U(\lie n^+)_\mathbb K^0V_{\lambda}=0, \quad\text{and}\quad V=U(\lie n^-)_\mathbb KV_\lambda.
\end{equation}
\item For every $\lambda\in P^+$ there exists a finite-dimensional irreducible $U(\lie g)_\mathbb K$-module, denoted by $V_\mathbb K(\lambda)$, which satisfies \eqref{e:hwrel}.
\item If $V$ satisfies \eqref{e:hwrel} for some $\lambda\in P^+$ and $V_\mu\ne 0$, then $\mu\le \lambda$.
\item The $U(\lie g)_\mathbb F$-module $V_\mathbb F(\lambda),\lambda\in P^+$, is isomorphic to $(V_\mathbb K(\lambda))^\mathbb F$.
\hfill\qedsymbol
\end{enumerate}
\end{thm}

We also recall the following basic facts on characters. Let $\cal C(\lie g)_\mathbb K$ be the category of finite-dimensional $U(\lie g)_\mathbb K$-modules. If $V$ is an object in $\cal C(\lie g)_\mathbb K$, the character of $V$ is the function $\ch(V):P\to\mathbb Z_+$, $\mu\mapsto \ch(V)_\mu:=\dim(V_\mu)$.  Observe that the last part of the above theorem implies that $\ch(V_\mathbb K(\lambda)) = \ch(V_\mathbb L(\lambda))$ for every $\lambda\in P^+$ and every field $\mathbb L$ having the same characteristic as $\mathbb K$.

It is convenient to also regard $\ch(V)$ as an element of the integral group ring $\mathbb Z[P]$ of $P$. It follows that there exists a ring homomorphism $\ch:{\rm Gr}(\cal C(\lie g)_\mathbb K) \to \mathbb Z[P]$ where ${\rm Gr}(\cal C(\lie g)_\mathbb K)$ is the Grothendieck ring of $\cal C(\lie g)_\mathbb K$. In particular, if $V$ and $W$ are objects in $\cal C(\lie g)_\mathbb K$ and $M$ is an extension of $V$ by $W$ we have
\begin{equation}\label{e:char}
\ch(M)=\ch(V)+\ch(W)\quad\text{and}\quad \ch(V\otimes W)=\ch(V)\ \ch(W).
\end{equation}
Using the first equality above, one easily devises an algorithm that proves the following corollary of Theorem \ref{t:cig}.

\begin{cor}\label{c:char}
Let $V$ be an object of $\cal C(\lie g)_\mathbb K$ and $\lambda\in P^+$. The Jordan-H\"older multiplicity $\mult^{\cal C(\lie g)_\mathbb K}_{V_\mathbb K(\lambda)}(V)$ is completely determined by $\ch(V)$.\hfill\qedsymbol
\end{cor}

In practice, in order to compute $\mult^{\cal C(\lie g)_\mathbb K}_{V_\mathbb K(\lambda)}(V)$ from a given $\ch(V)$, one must know $\ch(V_\mathbb K(\mu))$ for all $\mu\in P^+$ such that $\ch(V)_\mu\ne 0$. Character formulas are some of the most interesting results in combinatorial representation theory.  In the characteristic zero case, there are several formulas for computing $\ch(V_\mathbb K(\lambda))$ which can be found in \cite[Chapter VI]{humb}. In positive characteristic, the general solution for this problem is still open. We refer to \cite{and,don} for surveys of the related topics. See also \cite{fies,fiem} for the most recent developments.

\subsection{} In the next two subsections we review some results proved in \cite{jmnacf}. We begin with results on the classification of the finite-dimensional irreducible $U(\tlie g)_\mathbb K$-modules.

A finite-dimensional $U(\tlie g)_\mathbb K$-module $V$ is said to be a highest-quasi-$\ell$-weight module if there exists $\gb\omega\in\gb M_\mathbb F$ and a $\mathbb K[\gb\Lambda]$-submodule $\cal V$ of $V$ such that
\begin{equation}\label{e:ehwrel}
\cal V\cong\cal K(\gb\omega), \qquad U(\tlie n^+)_\mathbb K^0\cal V=0, \quad\text{and}\quad V=U(\tlie n^-)_\mathbb K\cal V.
\end{equation}
When $\deg(\gb\omega)=1$, we simply say $V$ is a highest-$\ell$-weight module. The conjugacy class $[\gb\omega]$ is said to be the highest (quasi)-$\ell$-weight of $V$ and $\cal V$ is said to be the highest-(quasi)-$\ell$-weight space of $V$.

Let $\wt:\cal P_\mathbb F\to P$ be the group homomorphism such that $\wt(\gb\omega)(h_i) =$ \linebreak $\deg(\gb\omega_i(u))$ for all $i\in I$ and all $\gb\omega\in\cal P_\mathbb F^+$. Here $\deg(\gb\omega_i(u))$ is the degree of $\gb\omega_i(u)$ as a polynomial.

\begin{thm}\label{t:ciwm}
Let $V$ be a finite-dimensional $U(\tlie g)_\mathbb K$-module. Then:
\begin{enumerate}
\item $V=\opl_{[\gb\varpi]\in\cal P_{\mathbb F,\mathbb K}}^{} V_{\gb\varpi}$ with $V_{\gb\varpi}$ a $\mathbb K[\gb\Lambda]$-submodule of $V$ isomorphic to a self-extension of $\cal K(\gb\varpi)$. Moreover, $V_\mu = \opl_{[\gb\varpi]:\wt(\gb\varpi)=\mu}^{} V_\gb\varpi$.
\item If $V$ is irreducible, $V$ is a highest-quasi-$\ell$-weight module with highest quasi-$\ell$-weight in $\cal P_{\mathbb F,\mathbb K}^+$.
\item Given $\gb\omega\in\cal P_\mathbb F^+$, there exists a universal finite-dimensional highest-quasi-$\ell$-weight $U(\tlie g)_\mathbb K$-module of highest quasi-$\ell$-weight $[\gb\omega]$, denoted by $W_\mathbb K(\gb\omega)$. Moreover, $W_\mathbb K(\gb\omega)$ has a unique irreducible quotient denoted by $V_\mathbb K(\gb\omega)$.
\item The isomorphism classes of $W_\mathbb K(\gb\omega)$ and $V_\mathbb K(\gb\omega)$ depend only on $[\gb\omega]\in\cal P_{\mathbb F,\mathbb K}^+$.
\item If $V$ is a highest-quasi-$\ell$-weight module of highest quasi-$\ell$-weight $[\gb\omega]\in\cal P_{\mathbb F,\mathbb K}^+$, then $V_\gb\varpi\ne 0$ only if there exists $\gb\varpi'\in[\gb\varpi]$ such that $\gb\varpi'\le \gb\omega$. Moreover, $V_\gb\omega$ is the highest-quasi-$\ell$-weight space of $V$.
\hfill\qedsymbol
\end{enumerate}
\end{thm}

A conjugacy class $[\gb\varpi]$ for which $V_\gb\varpi\ne 0$ is called a quasi-$\ell$-weight of $V$ and $V_\gb\varpi$ is the corresponding quasi-$\ell$-weight space. When $\deg(\gb\varpi)=1$, i.e., when $\gb\varpi\in\cal P_\mathbb K\subseteq\cal P_{\mathbb F,\mathbb K}$, we simply say $[\gb\varpi]$ is an $\ell$-weight of $V$. The module $W_\mathbb K(\gb\omega)$ is called the Weyl module of highest (quasi)-$\ell$-weight $[\gb\omega]$.

\subsection{} Now we recall results concerning base change. All the results of this note can be regarded as applications of part (b) the following theorem.

{\samepage
\begin{thm}\label{t:forms}
Let  $\gb\omega\in\cal P_\mathbb F^+$, $V=V_\mathbb F(\gb\omega)$, and $W=W_\mathbb F(\gb\omega)$.
\begin{enumerate}
\item If  $v$ is a nonzero vector in $V_\gb\omega$, then $U(\tlie g)_\mathbb Kv\cong V_\mathbb K(\gb\omega)$ and $\ch(V_\mathbb K(\gb\omega)) = \deg(\gb\omega)\ch(V_\mathbb F(\gb\omega))$. Similarly, if  $v$ is a nonzero vector in $W_\gb\omega$, then $U(\tlie g)_\mathbb Kv$ $\cong W_\mathbb K(\gb\omega)$ and $\ch(W_\mathbb K(\gb\omega)) = \deg(\gb\omega)$ $\ch(W_\mathbb F(\gb\omega))$.
\item $(V_\mathbb K(\gb\omega))^\mathbb F\cong\opl_{\gb\omega'\in[\gb\omega]}^{} V^\mathbb F(\gb\omega')$ where $V^\mathbb F(\gb\omega')$ is an indecomposable self-exten\-sion of $V_\mathbb F(\gb\omega')$ of length $\ideg(\gb\omega)$. Similarly, $(W_\mathbb K(\gb\omega))^\mathbb F\cong\opl_{\gb\omega'\in[\gb\omega]}^{} W^\mathbb F(\gb\omega')$ where $W^\mathbb F(\gb\omega')$ is an indecomposable self-extension of $W_\mathbb F(\gb\omega')$ of length $\ideg(\gb\omega)$.
\text{}\hfill\qedsymbol
\end{enumerate}
\end{thm}}

\section{Clebsch-Gordan Problem}\label{s:cg}

\subsection{}\label{ss:cgac} We now show how to solve the Clebsch-Gordan problem for the category $\cal C(\tlie g)_\mathbb K$ of finite-dimensional $U(\tlie g)_\mathbb K$-modules in terms of Clebsch-Gordan coefficients for $\cal C(\lie g)_\mathbb F$. In order to do so, we first reduce the problem to the category $\cal C(\tlie g)_\mathbb F$.
To simplify notation, given $\gb\omega\in\cal P_\mathbb F^+$ and a finite-dimensional $U(\tlie g)_\mathbb K$-module $V$, let
\begin{equation}
\mult^\mathbb K_{\gb\omega}(V)=\mult^{\cal C(\tlie g)_\mathbb K}_{V_\mathbb K(\gb\omega)}(V).
\end{equation}
Similarly, given $\gb\omega,\gb\varpi,\gb\pi\in\cal P_\mathbb F^+$, let
\begin{equation}
\mult^\mathbb K_{\gb\omega}(\gb\varpi,\gb\pi)=\mult^{\cal C(\tlie g)_\mathbb K}_{V_\mathbb K(\gb\omega)}(V_\mathbb K(\gb\varpi)\otimes V_\mathbb K(\gb\pi)).
\end{equation}

Notice that, if $\wt(\gb\omega)=\wt(\gb\varpi\gb\pi)$, then $\mult^\mathbb F_{\gb\omega}(\gb\varpi,\gb\pi)\le 1$ with equality holding if and only if $\gb\omega=\gb\varpi\gb\pi$.

Given $\gb\omega,\gb\varpi,\gb\pi\in\cal P_\mathbb F$, define
\begin{equation}
[\gb\omega:\gb\varpi,\gb\pi] = \{(\gb\varpi',\gb\pi')\in[\gb\varpi]\times[\gb\pi]:\gb\varpi'\gb\pi' =\gb\omega\}.
\end{equation}
Quite clearly the cardinality $|[\gb\omega:\gb\varpi,\gb\pi]|$ depends only on the conjugacy classes of $\gb\omega,\gb\varpi,\gb\pi$.

{\samepage
\begin{thm}\label{t:cg}Let $\gb\omega,\gb\varpi,\gb\pi\in\cal P_\mathbb F^+$. Then,
$$\mult^\mathbb K_{\gb\omega}(\gb\varpi,\gb\pi) = \frac{\ideg(\gb\varpi)\ \ideg(\gb\pi)}{\ideg(\gb\omega)}\sum_{\gb\varpi'\in[\gb\varpi]}\sum_{\gb\pi'\in[\gb\pi]}  \mult^\mathbb F_{\gb\omega}(\gb\varpi',\gb\pi').$$
In particular,  $\mult^\mathbb K_{\gb\omega}(\gb\varpi,\gb\pi)=\frac{\ideg(\gb\varpi)\ \ideg(\gb\pi)}{\ideg(\gb\omega)}\ |[\gb\omega:\gb\varpi,\gb\pi]|$ when $\wt(\gb\omega)=\wt(\gb\varpi\gb\pi)$.
\end{thm}}

\begin{proof}
Let $V = V_\mathbb K(\gb\varpi)\otimes_\mathbb K V_\mathbb K(\gb\pi)$. It follows from  Theorem \ref{t:forms}(b) together with Lemma \ref{l:compseries} applied to a composition series of $V$ in the category $\cal C(\tlie g)_\mathbb K$ that
\begin{equation*}
\mult^\mathbb K_{\gb\omega}(V)\deg(\gb\omega)  = \sum_{\gb\omega'\in[\gb\omega]} \mult^\mathbb F_{\gb\omega'}(V^\mathbb F) = |[\gb\omega]|\ \mult^\mathbb F_{\gb\omega}(V^\mathbb F).
\end{equation*}
On the other hand, it follows from Lemma \ref{l:tpfe} and Theorem \ref{t:forms}(b) that
\begin{equation*}
\mult^\mathbb F_{\gb\omega}(V^\mathbb F) = \ideg(\gb\varpi)\ \ideg(\gb\pi)\sum_{\gb\varpi'\in[\gb\varpi]}\sum_{\gb\pi'\in[\gb\pi]} \mult^\mathbb F_{\gb\omega}(\gb\varpi',\gb\pi').\qedhere
\end{equation*}
\end{proof}

As a consequence of the previous theorem, we have the following alternate proof of the nontrivial direction of the statement of \cite[Theorem 2.18]{jmnacf}.

\begin{cor}\label{c:cg}
Suppose $\gb\varpi,\gb\pi\in\cal P_\mathbb F^+$ are such that $V_\mathbb F(\gb\varpi)\otimes V_\mathbb F(\gb\pi)\cong V_\mathbb F(\gb\varpi\gb\pi)$ and $\deg(\gb\varpi\gb\pi) = \deg(\gb\varpi)\deg(\gb\pi)$. Then, $V_\mathbb K(\gb\varpi\gb\pi)\cong V_\mathbb K(\gb\varpi)\otimes V_\mathbb K(\gb\pi)$. 
\end{cor}

\begin{proof}
The second statement of Theorem \ref{t:cg} implies that $\mult^\mathbb K_{\gb\varpi\gb\pi}(\gb\varpi,\gb\pi)>0$. On the other hand, the hypotheses of the corollary together with Theorem \ref{t:forms}(a) imply $\dim(V_\mathbb K(\gb\varpi)\otimes V_\mathbb K(\gb\pi))=\dim(V_\mathbb K(\gb\varpi\gb\pi))$. 
\end{proof}

\subsection{}

Next, we recall how to rephrase the classification of finite-dimensional irreducible $U(\tlie g)_\mathbb F$-modules in terms of evaluation representations. We begin with the existence of the so-called evaluation maps (see \cite[Proposition 3.3]{jmhla}).

\begin{prop}\label{p:evmap}
For every $a\in \mathbb F^{\times}$, there exists a surjective algebra homomorphism ${\rm ev}_a: U(\tlie g)_\mathbb F\to U(\lie g)_\mathbb F$ called the evaluation map at $a$.\hfill\qedsymbol
\end{prop}

If $V$ is a $U(\lie g)_\mathbb F$-module, we denote by $V(a)$ the pull-back of $V$ by ${\rm ev}_a$.
Then, for $\lambda\in P^+$, $V_\mathbb F(\gb\omega_{\lambda,a})\cong (V_\mathbb F(\lambda))(a)$  (see \cite[\S 3B]{jmhla}). Observe that every $\gb\omega\in\cal P_\mathbb F^+$ can be uniquely written as a product of the form
\begin{equation}\label{e:factor}
\gb\omega = \prod_{j=1}^m \gb\omega_{\lambda_j,a_j} \quad\text{for some}\quad m\in\mathbb N, \lambda_j\in P^+, a_j\in\mathbb F^\times \text{ with } a_j\ne a_k\text{ if } j\ne k.
\end{equation}
Two elements $\gb\varpi,\gb\pi\in\cal P_\mathbb F^+$ are said to be relatively prime if $\gb\varpi_i(u)$ is relatively prime to $\gb\pi_j(u)$ for every $i,j\in I$.
The next proposition is an immediate consequence of \cite[Corollary 3.5 and Equation (3-13)]{jmhla}.

\begin{prop}\label{p:evrep}
If $\gb\varpi,\gb\pi\in\cal P_\mathbb F^+$ are relatively prime, then $V_\mathbb F(\gb\varpi)\otimes V_\mathbb F(\gb\pi)\cong V_\mathbb F(\gb\varpi\gb\pi)$. In particular, the finite-dimensional irreducible $U(\tlie g)_\mathbb F$-modules can be realized as tensor products of evaluations representations, i.e., if $\gb\omega = \prod\limits_{j=1}^m \gb\omega_{\lambda_j,a_j}$ as in \eqref{e:factor}, then $V_\mathbb F(\gb\omega)\cong \otm_{j=1}^m V_\mathbb F(\gb\omega_{\lambda_j,a_j})$.\vspace*{-17.5pt}\\ \text{}\hfill\qedsymbol
\end{prop}

The next proposition is easily established.

\begin{prop}
Let $\lambda,\mu\in P^+, V=V_\mathbb F(\lambda)\otimes V_\mathbb F (\mu)$, and let $0\subseteq V_1\subseteq V_2\subseteq\cdots\subseteq V_m=V$ be a composition series for $V$. Then $0\subseteq V_1(a)\subseteq V_2(a)\subseteq\cdots\subseteq V_m(a)=V(a)$ is a composition series for $V_\mathbb F(\gb\omega_{\lambda,a})\otimes V_\mathbb F (\gb\omega_{\mu,a})$.\hfill\qedsymbol
\end{prop}

Using the last two propositions, one easily reduces the Clebsch-Gordan problem for $\cal C(\tlie g)_\mathbb F$ to the one for $\cal C(\lie g)_\mathbb F$. By Corollary \ref{c:char}, solutions for $\cal C(\lie g)_\mathbb F$ can be obtained from the knowledge of $\ch(V_\mathbb F(\lambda)), \lambda\in P^+$.  The later is an open problem in general, as mentioned in subsection \ref{ss:fdg}. If $\mathbb F$ has characteristic zero there are several formulas for computing Clebsch-Gordan coefficients in $\cal C(\lie g)_\mathbb F$ (see for instance \cite{fh}). New insights into this problem were introduced by the advent of crystal bases (see \cite{bz,litlw,lityt} and references therein).

\section{$\ell$-Characters}\label{s:jh}

\subsection{} In this section, we consider the notion of $\ell$-characters -- the classical analogue of the $q$-characters defined by Frenkel and Reshetikhin in \cite{frqchar}. This is equivalent to studying $\mult^{\cal C(\gb\Lambda)_\mathbb K}_{\cal K(\gb\varpi)}(V_\mathbb K(\gb\omega))$ for every $\gb\varpi\in\cal P_\mathbb F$ and every $\gb\omega\in\cal P_\mathbb F^+$. Here,
 $\cal C(\gb\Lambda)_\mathbb K$ is the category of finite-dimensional $\mathbb K[\gb\Lambda]$-modules. Evidently, every object in  $\cal C(\tlie g)_\mathbb K$ is also an object in $\cal C(\gb\Lambda)_\mathbb K$ (recall also Theorem \ref{t:ciwm}(a)).

\begin{defn} Let $V$ be a finite-dimensional $U(\tlie g)_\mathbb K$-module.  The $\ell$-character of $V$ is the function $\chl(V):\cal P_{\mathbb F,\mathbb K}\to\mathbb Z_+$ given by $[\gb\varpi]\mapsto \chl(V)_{\gb\varpi}:=\dim(V_\gb\varpi)$. 
\end{defn}

Quite clearly
\begin{equation}
\mult^{\cal C(\gb\Lambda)_\mathbb K}_{\cal K(\gb\varpi)}(V) = \frac{\chl(V)_\gb\varpi}{\deg(\gb\varpi)}.
\end{equation}
Hence, studying $\chl(V)$ is equivalent to studying the Jordan-H\"older multiplicities of $V$ in the category $\cal C(\gb\Lambda)_\mathbb K$.

Using the partial order on $\cal P_\mathbb F$ defined in \S\ref{ss:l-lattices}, one proves  the following analogue of Corollary \ref{c:char}  as a corollary of Theorem \ref{t:ciwm}.

\begin{cor}\label{c:charl}
Let $V$ be an object of $\cal C(\tlie g)_\mathbb K$ and $\gb\omega\in\cal P_\mathbb F^+$. The Jordan-H\"older multiplicity $\mult^{\mathbb K}_{\gb\omega}(V)$ is completely determined by $\chl(V)$.\hfill\qedsymbol
\end{cor}

A similar coment as the one we made after Corollary \ref{c:char} is in place here, as well. Namely, the knowledge of $\chl(V_\mathbb K(\gb\varpi))$ for every $\gb\varpi\in\cal P_\mathbb F^+$ is required in order to carry out, in practice, the computation of $\mult^{\mathbb K}_{\gb\omega}(V)$ from a given $\chl(V)$.

\subsection{} The results of \S\ref{ss:cgac} and \eqref{e:char} enable us to reduce the problem of computing $\chl(V_\mathbb F(\gb\omega))$ for $\gb\omega\in\cal P_\mathbb F^+$ to those of computing  Clebsch-Gordan coefficients for $\cal C(\lie g)_\mathbb F$ as well as $\ch(V_\mathbb F(\lambda))$, $\lambda\in P^+$. The next theorem expresses $\chl(V_\mathbb K(\gb\omega))$ and $\chl(W_\mathbb K(\gb\omega))$ in terms of $\chl(V_\mathbb F(\gb\omega))$  and $\chl(W_\mathbb F(\gb\omega))$, respectively.

\begin{thm}\label{t:charl} For every $\gb\omega\in\cal P_\mathbb F^+$ and every $\gb\varpi\in\cal P_\mathbb F$ we have
$$\chl(V_\mathbb K(\gb\omega))_\gb\varpi = \deg(\gb\omega)\sum_{\gb\varpi'\in[\gb\varpi]} \chl(V_\mathbb F(\gb\omega))_{\gb\varpi'}$$
and
\begin{equation*}
\chl(W_\mathbb K(\gb\omega))_\gb\varpi = \deg(\gb\omega)\sum_{\gb\varpi'\in[\gb\varpi]} \chl(W_\mathbb F(\gb\omega))_{\gb\varpi'}.
\end{equation*}
\end{thm}

\begin{proof}
Set $V=V_\mathbb K(\gb\omega)$. By Theorem \ref{t:irpa}(c) and  Lemma \ref{l:compseries} applied to a composition series of $V$ in the category $\cal C(\gb\Lambda)_\mathbb K$ we have
$$\chl(V)_\gb\varpi = \sum_{\gb\varpi'\in[\gb\varpi]}\chl(V^\mathbb F)_{\gb\varpi'}.$$
On the other hand, it is not difficult to see from the proof of  Theorem \ref{t:forms}(b) (see the proof of Theorem 2.12 and \S A.3 of \cite{jmnacf})  that every $g\in{\rm Aut}(\mathbb F/\mathbb K)$ induces an isomorphism of $U(\tlie g)_\mathbb K$-modules $V_\mathbb F(\gb\omega)\to V_\mathbb F(g(\gb\omega))$. In particular, for every $\gb\omega'\in[\gb\omega]$ and every $\gb\varpi\in\cal P_\mathbb F$, we have
$$\sum_{\gb\varpi'\in[\gb\varpi]}\chl(V_\mathbb F(\gb\omega))_{\gb\varpi'} = \sum_{\gb\varpi'\in[\gb\varpi]}\chl(V_\mathbb F(\gb\omega'))_{\gb\varpi'}.$$
Since, by Theorem \ref{t:forms}(b), $V^\mathbb F$ is a $U(\tlie g)_\mathbb F$-module of length $\deg(\gb\omega)$ whose irreducible constituents are of the form $V_\mathbb F(\gb\omega')$ with $\gb\omega'\in[\gb\omega]$, it follows that
$$\sum_{\gb\varpi'\in[\gb\varpi]}\chl(V^\mathbb F)_{\gb\varpi'}=\deg(\gb\omega)\sum_{\gb\varpi'\in[\gb\varpi]}\chl(V_\mathbb F(\gb\omega))_{\gb\varpi'}.$$
This proves the first formula. For the second one, we observe that all the above steps can be carried out with $W_\mathbb K(\gb\omega)$ in place of $V_\mathbb K(\gb\omega)$ with the corresponding obvious modifications.
\end{proof}

\subsection{}\label{ss:cgec} Due to Corollary \ref{c:charl}, the formula of Theorem \ref{t:cg} can, in principle, be recovered from $\chl(V)$ with $V=V_\mathbb K(\gb\varpi)\otimes V_\mathbb K(\gb\pi)$. In the algebraically closed case, the notion of $\ell$-character can be reinterpreted as a ring homomorphism $\chl:{\rm Gr}(\cal C(\tlie g)_\mathbb F) \to \mathbb Z[\cal P_\mathbb F]$ where ${\rm Gr}(\cal C(\tlie g)_\mathbb F)$ is the Grothendieck ring of $\cal C(\tlie g)_\mathbb F$ and $\mathbb Z[\cal P_\mathbb F]$ is the  integral group ring of $\cal P_\mathbb F$. In particular, if $V$ and $W$ are objects in $\cal C(\tlie g)_\mathbb F$ and $M$ is an extension of $V$ by $W$ we have
\begin{equation}\label{e:charl}
\chl(M)=\chl(V)+\chl(W)\ \ \text{and}\ \ \chl(V\otimes W)=\chl(V)\ \chl(W).
\end{equation}
The second formula follows easily from the comultiplication of the elements $\Lambda_{i,r}$ (cf. the proof of \cite[Corollary 3.5]{jmhla}).
In the non-algebraically closed case, $\cal P_{\mathbb F,\mathbb K}$ is not a group, but we can still form the free $\mathbb Z$-module $\mathbb Z[\cal P_{\mathbb F,\mathbb K}]$ and reinterpret $\ell$-character as a group homomorphism from the Grothendieck group of $\cal C(\tlie g)_\mathbb K$ to  $\mathbb Z[\cal P_{\mathbb F,\mathbb K}]$. We again have $\chl(M)=\chl(V)+\chl(W)$ whenever $M$ is an extension of $V$ by $W$. Although multiplication is not defined on $\mathbb Z[\cal P_{\mathbb F,\mathbb K}]$ and, hence, the second formula of \eqref{e:charl} does not make sense anymore, we have a replacement formula given  by Theorem \ref{t:tpcharl} below in the case of irreducible modules.

Given $\eta\in\mathbb Z[\cal P_\mathbb F]$, say
$$\eta = \sum_{\gb\varpi\in\cal P_\mathbb F} \eta(\gb\varpi)\ \gb\varpi \qquad\text{with}\quad \eta(\gb\varpi)\in\mathbb Z,$$
define
\begin{equation}\label{e:[charl]}
[\eta] = \sum_{[\gb\varpi]\in\cal P_{\mathbb F,\mathbb K}} \eta([\gb\varpi])\ [\gb\varpi]\in\mathbb Z[\cal P_{\mathbb F,\mathbb K}] \qquad\text{where}\quad \eta([\gb\varpi]) = \sum_{\gb\varpi'\in[\gb\varpi]} \eta(\gb\varpi').
\end{equation}

{\samepage
\begin{thm}\label{t:tpcharl}
For every $\gb\varpi,\gb\pi\in\cal P_\mathbb F^+$,
\begin{gather*}
\chl(V_\mathbb K(\gb\varpi)\otimes V_\mathbb K(\gb\pi)) = \\
\left[ \ideg(\gb\varpi)\ \ideg(\gb\pi)  \sum_{\gb\varpi'\in[\gb\varpi]}\sum_{\gb\pi'\in[\gb\pi]} \chl(V_\mathbb F(\gb\varpi'))\ \chl(V_\mathbb F(\gb\pi')) \right].
\end{gather*}
\end{thm}

\begin{proof}
Similar to the proofs of Theorems \ref{t:cg} and \ref{t:charl}. We omit the details (see also the proof of Theorem \ref{t:mweyl} below).
\end{proof}}

Notice that the formulas of Theorem \ref{t:charl} may be rewritten using \eqref{e:[charl]} as
\begin{equation*}
\chl(V_\mathbb K(\gb\omega)) = \deg(\gb\omega)[\chl(V_\mathbb F(\gb\omega))]
\end{equation*}
and
\begin{equation*}
\chl(W_\mathbb K(\gb\omega)) = \deg(\gb\omega)[\chl(W_\mathbb F(\gb\omega))].
\end{equation*}
\vspace{-5pt}

\section{Jordan-H\"older Multiplicities for Weyl Modules}\label{s:jhw}

\subsection{}  We now address the problem of computing $\mult^{\mathbb K}_{\gb\varpi}(W_\mathbb K(\gb\omega))$ for every $\gb\omega,\gb\varpi\in\cal P_\mathbb F^+$.
 In light of Corollary \ref{c:charl}, it is in principle possible to combine the two formulas of Theorem \ref{t:charl} to compute $\mult^{\mathbb K}_{\gb\varpi}(W_\mathbb K(\gb\omega))$ in terms of $\chl(W_\mathbb F(\gb\omega))$. However, it is actually easier to proceed similarly to the proofs of Theorems \ref{t:cg} and \ref{t:charl} to obtain the following.

\begin{thm}\label{t:mweyl}
If $\gb\omega,\gb\varpi\in\cal P_\mathbb F^+$,
\begin{align*}
\mult^{\mathbb K}_{\gb\varpi}(W_\mathbb K(\gb\omega))  & = \frac{\ideg(\gb\omega)}{\ideg(\gb\varpi)}\sum_{\gb\omega'\in[\gb\omega]} \mult^{\mathbb F}_{\gb\varpi}(W_\mathbb F(\gb\omega'))\\
&=  \frac{\deg(\gb\omega)}{\deg(\gb\varpi)}\sum_{\gb\varpi'\in[\gb\varpi]} \mult^{\mathbb F}_{\gb\varpi'}(W_\mathbb F(\gb\omega)).
\end{align*}
\end{thm}

\begin{proof}
Set $V=W_\mathbb K(\gb\omega)$. It follows from the first part of Theorem \ref{t:forms}(b) together with Lemma \ref{l:compseries} applied to a composition series of $V$ in the category $\cal C(\tlie g)_\mathbb K$ that
\begin{equation*}
\mult^\mathbb K_{\gb\varpi}(V)\deg(\gb\varpi)  = \sum_{\gb\varpi'\in[\gb\varpi]} \mult^\mathbb F_{\gb\varpi'}(V^\mathbb F) = |[\gb\varpi]|\ \mult^\mathbb F_{\gb\varpi}(V^\mathbb F).
\end{equation*}
On the other hand, the second part of Theorem \ref{t:forms}(b) gives
\begin{equation*}
\mult^\mathbb F_{\gb\pi}(V^\mathbb F) = \ideg(\gb\omega) \sum_{\gb\omega'\in[\gb\omega]}\mult^\mathbb F_{\gb\pi}(W_\mathbb F(\gb\omega'))\qquad\text{for all}\qquad \gb\pi\in\cal P_\mathbb F^+.
\end{equation*}
This completes the proof of the first equality.

To prove the second equality, one proceeds similarly to the proof of Theorem \ref{t:charl} to get
$$\sum_{\gb\varpi'\in[\gb\varpi]}\mult^\mathbb F_{\gb\varpi'}(W_\mathbb F(\gb\omega)) = \sum_{\gb\varpi'\in[\gb\varpi]}\mult^\mathbb F_{\gb\varpi'}(W_\mathbb F(\gb\omega'))\qquad\text{for all}\qquad \gb\omega'\in[\gb\omega].$$
Now the second part of Theorem \ref{t:forms}(b) gives
\begin{equation*}
\sum_{\gb\varpi'\in[\gb\varpi]} \mult^\mathbb F_{\gb\varpi'}(V^\mathbb F) = \deg(\gb\omega)\sum_{\gb\varpi'\in[\gb\varpi]}\mult^\mathbb F_{\gb\varpi'}(W_\mathbb F(\gb\omega)).\qedhere
\end{equation*}
\end{proof}

\subsection{}
We now review the results leading to a method  for reducing the computation of  $\mult^{\mathbb F}_{\gb\varpi}(W_\mathbb F(\gb\omega))$ to the computation of certain multiplicities in the category $\cal C(\lie g)_\mathbb F$.

We observe that all the results of this subsection in the positive characteristic setting depend on the conjecture of \cite{jmhla} which says that all Weyl modules for $U(\tlie g)_\mathbb F$, when $\mathbb F$ is of positive characteristic, can be obtained by reduction modulo $p$ from appropriate Weyl modules in characteristic zero. This statement can be regarded as an analogue of a conjecture of Chari and Pressley \cite{cpweyl} saying that all Weyl modules for $U(\tlie g)_\mathbb C$ can be obtained as classical limits of appropriate Weyl modules for quantum affine algebras. Chari-Pressley's conjecture has been recently proved for $\lie g$ of type A in \cite{cl} and for simply laced $\lie g$ in \cite{fol}. Moreover, it has been pointed out by Nakajima that the general case follows from the global and crystal basis theory developed in \cite{bn,kascb,kaslz,nakq,nake} (see \cite{cl,fol} for brief summaries of Nakajima's arguments). We expect the conjecture of \cite{jmhla} can be proved using similar arguments.

We start with the following proposition.

\begin{prop}\label{p:cweyl}
Let $\lambda\in P^+$ and $a\in\mathbb F^\times$. If $0\subseteq V_1\subseteq V_2\subseteq\cdots\subseteq V_m=V$ is a composition series for $V=W_\mathbb F(\gb\omega_{\lambda,a})$ as a $U(\lie g)_\mathbb F$-module, then $0\subseteq V_1(a)\subseteq V_2(a)\subseteq\cdots\subseteq V_m(a)$ is a composition series for $W_\mathbb F(\gb\omega_{\lambda,a})$ as a $U(\tlie g)_\mathbb F$-module.
\end{prop}

\begin{proof}
In characteristic zero this is proved in \cite[Proposition 3.3]{cmsc}. The positive characteristic case follows from the characteristic zero case using the conjecture of \cite{jmhla} (cf.  \cite[Proposition 4.11]{jmhla}).
\end{proof}

Hence, in the spirit of Corollary \ref{c:char}, in order to compute $\mult^{\mathbb F}_{\gb\varpi}(W_\mathbb F(\gb\omega))$ when $\gb\omega = \gb\omega_{\lambda,a}$, it suffices to know $\ch(W_\mathbb F(\gb\omega_{\lambda,a}))$.
In characteristic zero, this character may be computed as a by-product of the proofs of Chari-Pressley's conjecture mentioned above. Namely, characters are unchanged by taking classical limits. Moreover, it follows from \cite{bn,cmqc} that the Weyl modules for quantum affine algebras may be realized as tensor products of the so-called fundamental representations. Hence, by the multiplicative property of characters with respect to tensor products, it is left to compute the characters of the later representations. For this, and other related problems, we refer to \cite{cmfund,hermin} and references therein. Finally, the positive characteristic case then follows from the conjecture of \cite{jmhla}, since the conjecture implies the character of Weyl modules is unchanged by the reduction modulo $p$ process studied in \cite{jmhla}. In particular, $\ch(W_\mathbb F(\gb\omega_{\lambda,a}))$ depends only on $\lambda$ (and hence on $\lie g$), but not on $\mathbb F$ or $a$.

A method for computing $\mult^{\mathbb F}_{\gb\varpi}(W_\mathbb F(\gb\omega))$ for general $\gb\omega$ is now obtained by combining the special case $\gb\omega=\gb\omega_{\lambda,a}$ with the solutions of the Clebsch-Gordan problem for the category $\cal C(\tlie g)_\mathbb F$ and the next theorem.

\begin{thm}
Let $\gb\omega\in\cal P_\mathbb F^+$ and write $\gb\omega = \prod\limits_{j=1}^m \gb\omega_{\lambda_j,a_j}$ as in \eqref{e:factor}. Then $W_\mathbb F(\gb\omega)\cong \otm_{j=1}^m W_\mathbb F(\gb\omega_{\lambda_j,a_j})$.\hfill\qedsymbol
\end{thm}

The theorem above was proved in \cite{cpweyl} for the characteristic zero case and follows again from the conjecture of \cite{jmhla} for the positive characteristic setting.

\bibliographystyle{amsplain}

\end{document}